\let\PaperKernelRefstepcounter\refstepcounter
\PassOptionsToPackage{table}{xcolor}

 \newif\ifarxivsubmission
\arxivsubmissionfalse
\ifdefined\UseArxivSubmission
  \arxivsubmissiontrue
\fi

  \arxivsubmissiontrue

\ifarxivsubmission
  \documentclass[a4paper]{article}
  \newcommand{\PaperTitle}{Polynomial Interpolation of a Vector Field on a Convex Polyhedral Domain}
\newcommand{\PaperShortTitle}{Polynomial Interpolation of a Vector Field}
\newcommand{\PaperShortAuthors}{J.~Chu and S.~Kaji}

\newcommand{\PaperAuthorAName}{Junyan Chu}
\newcommand{\PaperAuthorAEmail}{chujy626@gmail.com}
\newcommand{\PaperAuthorAUrl}{https://sites.google.com/view/junyan-chu/}

\newcommand{\PaperAuthorBName}{Shizuo Kaji}
\newcommand{\PaperAuthorBEmail}{kaji.shizuo.7r@kyoto-u.ac.jp}
\newcommand{\PaperAuthorBUrl}{https://www.skaji.org}

\newcommand{\PaperAffiliation}{Graduate School of Science, Kyoto University, Kyoto, Japan}
\newcommand{\PaperFunding}{This research was partially supported by JST Moonshot R\&D Grant Number JPMJMS2021, and KAKENHI, Grant-in-Aid for Scientific Research (B) 25K00921 and (S) 25H00399.}

\newcommand{\PaperAbstract}{
We develop a method for reconstructing polynomial vector fields from discrete velocity samples on a convex polyhedral domain under an exact no-penetration boundary condition.
For a prescribed degree bound, the method computes a polynomial vector field that fits the sampled data in the least-squares sense while satisfying the tangency condition identically on every boundary facet.
The central ingredient is an explicit algebraic characterization of the space of polynomial vector fields tangent to the boundary, obtained from the theory of logarithmic derivations of hyperplane arrangements.
This characterization reduces the constrained reconstruction problem to finite-dimensional linear algebra.
We also discuss extensions incorporating additional linear differential constraints, such as incompressibility, and piecewise polynomial constructions on non-convex polyhedral domains with prescribed smoothness across cell interfaces.
}

\newcommand{\PaperKeywords}{vector field reconstruction, polynomial interpolation, no-penetration boundary condition, hyperplane arrangement, piecewise polynomial}
\newcommand{\PaperMSCcodes}{65D05, 41A10, 41A15, 13P10, 52C35}

\usepackage[numbers]{natbib}
\usepackage{amsmath,amsthm}
\usepackage{hyperref}
\usepackage{cleveref}
\usepackage[margin=2cm]{geometry}
\theoremstyle{plain}
\newtheorem{theorem}{Theorem}[section]
\newtheorem{lemma}[theorem]{Lemma}

\newtheorem{proposition}[theorem]{Proposition}

\theoremstyle{definition}

\newtheorem{example}[theorem]{Example}
\newtheorem{problem}[theorem]{Problem}

\theoremstyle{remark}
\newtheorem{remark}[theorem]{Remark}

\crefname{section}{Section}{Sections}
\crefname{definition}{Definition}{Definitions}
\crefname{theorem}{Theorem}{Theorems}
\crefname{lemma}{Lemma}{Lemmas}
\crefname{corollary}{Corollary}{Corollaries}
\crefname{proposition}{Proposition}{Propositions}
\crefname{remark}{Remark}{Remarks}
\crefname{example}{Example}{Examples}
\crefname{problem}{Problem}{Problems}
\crefname{claim}{Claim}{Claims}
\crefname{hypothesis}{Hypothesis}{Hypotheses}

\usepackage{etoolbox}
\usepackage{needspace}
\usepackage{enumitem}
\makeatletter
\newcommand{\@currentthmtype}{}
\AtBeginDocument{%
  \let\cref@real@label\label
  \def\label{\@ifnextchar[\cref@real@label\paper@label}%
  \def\paper@label#1{%
    \ifx\@currentthmtype\@empty
      \cref@real@label{#1}%
    \else
      \ifdefstring{\@currentcounter}{theorem}{%
        \expandafter\label@optarg\expandafter[\@currentthmtype]{#1}%
      }{\cref@real@label{#1}}%
    \fi}%
}
\newcommand{\TrackThmType}[1]{%
  \AtBeginEnvironment{#1}{\gdef\@currentthmtype{#1}}%
  \AtEndEnvironment{#1}{\gdef\@currentthmtype{}}%
}
\makeatother
\TrackThmType{theorem}
\TrackThmType{lemma}
\TrackThmType{corollary}
\TrackThmType{proposition}
\TrackThmType{definition}
\TrackThmType{example}
\TrackThmType{problem}
\TrackThmType{remark}
\TrackThmType{hypothesis}
\TrackThmType{claim}

\AtBeginEnvironment{thebibliography}{%
  \Needspace{5\baselineskip}\phantomsection\interlinepenalty=10000\relax}

\newenvironment{keywords}{%
  \par\vspace{0.5em}
  \noindent\textbf{Keywords:} %
}{\par\vspace{0.5em}}

\newenvironment{MSCcodes}{%
  \par\vspace{0.5em}
  \noindent\textbf{MSC2020:} %
}{\par\vspace{0.5em}}

\title{\PaperTitle}
\author{%
  \PaperAuthorAName\thanks{\PaperAffiliation. E-mail: \texttt{\PaperAuthorAEmail}. URL: \url{\PaperAuthorAUrl}.}
  \and
  \PaperAuthorBName\thanks{\PaperAffiliation. E-mail: \texttt{\PaperAuthorBEmail}. URL: \url{\PaperAuthorBUrl}.}
}
\date{\vspace{-1.5em}}

\newenvironment{paperkeywords}{\begin{keywords}}{\end{keywords}}
\newcommand{\printpapermsc}{%
  \begin{MSCcodes}
  \PaperMSCcodes
  \end{MSCcodes}
}
\newcommand{\usepaperbibliographystyle}{\bibliographystyle{plainnat}}

\else
  \documentclass[review]{siamart251216}

\newsiamremark{example}{Example}
\newsiamremark{problem}{Problem}
\newsiamremark{remark}{Remark}
\newsiamremark{hypothesis}{Hypothesis}
\newsiamremark{claim}{Claim}
\numberwithin{theorem}{section}

\crefname{section}{Section}{Sections}
\crefname{definition}{Definition}{Definitions}
\crefname{theorem}{Theorem}{Theorems}
\crefname{lemma}{Lemma}{Lemmas}
\crefname{corollary}{Corollary}{Corollaries}
\crefname{proposition}{Proposition}{Propositions}
\crefname{remark}{Remark}{Remarks}
\crefname{example}{Example}{Examples}
\crefname{problem}{Problem}{Problems}
\crefname{claim}{Claim}{Claims}
\crefname{hypothesis}{Hypothesis}{Hypotheses}

\usepackage{etoolbox}
\usepackage{needspace}
\usepackage{enumitem}
\makeatletter
\newcommand{\@currentthmtype}{}
\AtBeginDocument{%
  \let\cref@real@label\label
  \def\label{\@ifnextchar[\cref@real@label\paper@label}%
  \def\paper@label#1{%
    \ifx\@currentthmtype\@empty
      \cref@real@label{#1}%
    \else
      \ifdefstring{\@currentcounter}{theorem}{%
        \expandafter\label@optarg\expandafter[\@currentthmtype]{#1}%
      }{\cref@real@label{#1}}%
    \fi}%
}
\newcommand{\TrackThmType}[1]{%
  \AtBeginEnvironment{#1}{\gdef\@currentthmtype{#1}}%
  \AtEndEnvironment{#1}{\gdef\@currentthmtype{}}%
}
\makeatother
\TrackThmType{theorem}
\TrackThmType{lemma}
\TrackThmType{corollary}
\TrackThmType{proposition}
\TrackThmType{definition}
\TrackThmType{example}
\TrackThmType{problem}
\TrackThmType{remark}
\TrackThmType{hypothesis}
\TrackThmType{claim}

\AtBeginEnvironment{thebibliography}{%
  \Needspace{5\baselineskip}\phantomsection\interlinepenalty=10000\relax}

\makeatletter
\IfFormatAtLeastTF{2024-11-01}{%
  \let\cref@old@refstepcounter\PaperKernelRefstepcounter
  \let\refstepcounter@noarg\PaperKernelRefstepcounter
}{%
  \def\refstepcounter@noarg#1{\H@refstepcounter{#1}}%
  \def\cref@old@refstepcounter#1{\H@refstepcounter{#1}}%
}
\makeatother

\headers{\PaperShortTitle}{\PaperShortAuthors}

\title{\PaperTitle%
\thanks{\funding{\PaperFunding}}}

\author{%
\PaperAuthorAName\thanks{\PaperAffiliation\ (\email{\PaperAuthorAEmail}).}
\and
\PaperAuthorBName\thanks{\PaperAffiliation\ (\email{\PaperAuthorBEmail}).}%
}

\newenvironment{paperkeywords}{\begin{keywords}}{\end{keywords}}
\newcommand{\printpapermsc}{\begin{AMS}\PaperMSCcodes\end{AMS}}
\newcommand{\usepaperbibliographystyle}{\bibliographystyle{siamplain}}

\fi

\usepackage{amsfonts,amssymb}
\usepackage{graphicx}
\usepackage{placeins}
\usepackage{epstopdf}
\usepackage{algorithm, algorithmicx, algpseudocode}
\usepackage{tikz}
\usetikzlibrary{arrows.meta,calc}
\usepackage{amsopn}

\ifpdf
  \DeclareGraphicsExtensions{.eps,.pdf,.png,.jpg}
\else
  \DeclareGraphicsExtensions{.eps}
\fi

\crefname{problem}{Problem}{Problems}

\DeclareMathOperator{\Syz}{Syz}

\newcommand{\Poly}[1]{{\mathrm{Poly}({#1})}}
\newcommand{\Pt}[1]{{\mathrm{Poly_{\partial}}({#1})}}
\newcommand{\Pext}[1]{{\mathrm{Poly}_{\mathrm{ext}}({#1})}}
\newcommand{\PPt}[2]{{\mathrm{PPoly}^{#1}_{\partial}({#2})}}
\newcommand{\error}{\mathcal{E}}
\newcommand{\relerr}{\mathcal{E}_{\mathrm{rel}}}
\DeclareMathOperator{\leak}{leak}

\newcommand{\Dt}[1]{{\mathrm{DivFree_{\partial}}({#1})}}
\newcommand{\Rt}[1]{{\mathrm{RotFree_{\partial}}({#1})}}
\newcommand{\Ht}[1]{{\mathrm{Harm_{\partial}}({#1})}}

\newcommand{\mP}{\mathcal{P}}
\newcommand{\mQ}{\mathcal{Q}}
\newcommand{\Ext}{\mathop{\mathrm{ext}}}
\newcommand{\R}{\mathbb{R}}
\newcommand{\Rx}{\mathbb{R}[x_1,\ldots,x_d]}

\newcommand{\Obs}{\mathcal{O}}
\newcommand{\hh}{\widehat h}
\newcommand{\hmp}{\widehat{\mathcal{P}}}

\begin{document}
\maketitle

\begin{abstract}
\PaperAbstract
\end{abstract}

\begin{paperkeywords}
\PaperKeywords
\end{paperkeywords}

\printpapermsc

\section{Introduction}\label{sec:intro}
We consider reconstruction of a vector field inside a polyhedral domain from sparse samples under the no-penetration boundary condition that requires the velocity field to be tangent to each boundary facet.
More precisely, let $\mP\subset\R^d$ be a full-dimensional convex polyhedron.
For a prescribed degree bound $k$, we seek a polynomial vector field $\xi$ of degree at most $k$ that
\begin{enumerate}[label=(\roman*)]
\item minimizes the approximation error with respect to the given velocity samples; and
\item satisfies the no-penetration condition exactly along the boundary $\partial\mP$.
\end{enumerate}
Such a formulation arises naturally in the reconstruction of fluid or particle velocities in rigid containers with sparse measurements.
Exact enforcement of the boundary condition prevents the reconstructed field from introducing artificial flux through the container wall and is therefore important for physical consistency.
The condition enforces impermeability while permitting tangential motion; it is the kinematic part of a free-slip boundary condition.

Once a basis for the space of admissible polynomial vector fields is available, the reconstruction problem reduces to a finite-dimensional linear least-squares problem.
The main observation underlying our approach is that the no-penetration boundary condition can be expressed as an ideal membership problem: each facet equation must divide the corresponding normal component of the field.
This connects the reconstruction problem with the theory of hyperplane arrangements, and provides an explicit computational description of the admissible model space.

More precisely, our construction proceeds as follows:
\begin{enumerate}
\item Represent the polyhedron $\mP$ as the intersection of finitely many affine half-spaces.
\item Homogenize the affine hyperplanes supporting its facets and add the homogenizing hyperplane $x_0=0$, obtaining a family $\hmp$ of hyperplanes through the origin in $\R^{d+1}$.
\item Relate polynomial vector fields of degree at most $k$ on $\R^d$ satisfying the no-penetration condition on $\partial\mP$ to homogeneous polynomial vector fields of degree $k$ on $\R^{d+1}$ that are tangent to every hyperplane of $\hmp$ and have vanishing component in the homogenizing direction.
\item Describe this space through syzygies among the partial derivatives of the product of the hyperplane equations.
\item Compute generators using Gr\"obner-basis methods, in particular Schreyer's algorithm, extract a basis at the prescribed degree, and map it back to the original domain.
The fitting problem then becomes an ordinary linear least-squares problem in the resulting coefficients.
\end{enumerate}

An abstract identification of the spaces does not by itself supply a fitting basis. We therefore construct explicit degree-preserving isomorphisms that convert computed syzygies into polynomial fields on the original domain.

The same framework accommodates several extensions.
Additional linear differential constraints can be imposed by adding linear conditions on the coefficients of the tangential basis.
An important example is the divergence-free condition, which enforces incompressibility.
The construction also extends to non-convex polyhedral domains by decomposing them into convex cells.
Tangency is then imposed on exterior facets, while derivatives across internal interfaces are matched up to a prescribed order.
This yields piecewise polynomial vector fields with a specified degree of smoothness across the cell decomposition.

The principal contribution of this work is an explicit algebraic and computational framework for constructing polynomial vector fields that satisfy polyhedral no-penetration conditions identically rather than only at sampled boundary points.
The connection with logarithmic derivation modules makes it possible to compute the constrained polynomial space systematically and thereby separate the exact geometric constraint from the subsequent data-fitting problem.
More broadly, the construction illustrates how methods from the algebraic theory of hyperplane arrangements can be used in constrained approximation problems for vector fields.

The remainder of the paper is organized as follows.
\Cref{sec:related} reviews related work.
\Cref{sec:problem} fixes notation, characterizes tangency by divisibility, and states the fitting problem.
\Cref{sec:interpolation} contains the main result, the identification of the admissible space with a module of syzygies, together with the resulting dimension formulas.
\Cref{sec:algorithm} gives the fitting algorithm and proves termination of degree selection for boundary-compatible velocity data, including divergence-free fitting in dimension at least two. It also treats additional linear differential constraints. \Cref{sec:piecewise} extends the construction and the termination guarantee to piecewise fields of fixed smoothness on non-convex domains, and \Cref{sec:examples} reports numerical examples.

Our code is available on GitHub.\footnote{\url{https://github.com/shizuo-kaji/HyperPlaneArrangementSAGE}}

\section{Related work}\label{sec:related}

Vector field interpolation and reconstruction from scattered observations have a long history, with applications in meteorology~\cite{schaefer1979interpolation} and in motor control, where fields are expanded in basis fields~\cite{Mussa-Ivaldi1992}. Methods for reconstruction from sparse samples~\cite{Lage} and radial-basis-function approaches~\cite{dodu2002vectorial} provide flexible approximation spaces for irregularly sampled data. Differential constraints can be incorporated directly into the approximation space through matrix-valued kernels~\cite{NarcowichWard1994}. In particular, \cite{FarrellGillowWendland2017} develop a multilevel method for divergence-free interpolation. For data defined on staggered grids, \cite{SCHROEDER} construct local, continuous, divergence-free piecewise polynomial interpolants. Boundary conditions in such approximation methods are often enforced through additional constraints or penalty terms, in which case they are generally satisfied only at selected points or approximately.

An alternative strategy is to incorporate boundary conditions directly into the trial space. For example, \cite{SukumarSrivastava2022} use distance functions to construct neural-network approximations of PDE solutions that satisfy prescribed boundary conditions by construction. This follows the same general principle as the present work: the admissible representation is chosen so that the boundary condition is satisfied identically, rather than imposed subsequently during optimization. Tangency conditions on convex polyhedra also arise in \cite{Robbins_2004}, which studies unit-vector fields tangent to the faces of a polyhedron. Generalized barycentric coordinates~\cite{Floater2015} provide another geometry-dependent framework for interpolation on polytopes, primarily for scalar functions or data prescribed at vertices.

Our construction uses algebraic descriptions of tangency. Saito introduced logarithmic derivations to describe vector fields tangent to hypersurfaces~\cite{saito1980theory}; for hyperplane arrangements, logarithmic derivations encode tangency to each hyperplane through divisibility conditions~\cite{OrlikTerao1992}. In the present setting, homogenization relates logarithmic derivations of the central hyperplane arrangement to bounded-degree polynomial vector fields satisfying the no-penetration condition on the original affine polyhedron. The resulting spaces can be described in terms of syzygies and computed using Gr\"obner-basis techniques, including Schreyer's algorithm~\cite{Schreyer1991,Eisenbud}.

\section{Tangential polynomial field}\label{sec:problem}
Let $R=\Rx$, with $R_{\le k}$ its polynomials of total degree at most $k$ and $R_{\le k}=\{0\}$ for $k<0$. Throughout, $d\ge1$, and $k$ is a nonnegative integer. We consider a full-dimensional convex polyhedron
\[
\mP=\{x\in\R^d \mid h_i(x)\le0,\ 1\le i\le m\},\qquad
h_i(x)=\langle\alpha_i,x\rangle+\ell_i,
\]
where face normals $\alpha_i\ne0$ and the inequalities are irredundant. Thus $F_i=\mP\cap H_i$ is a facet with supporting hyperplane $H_i=\{h_i=0\}$. Bounded polyhedra are called polytopes. The algebraic constructions also apply to unbounded polyhedra and to $\mP=\R^d$ when $m=0$.

Since $\mP$ has nonempty interior, each polynomial field on $\mP$ has a unique polynomial extension to $\R^d$. We identify it with this extension and write $\Poly{\mP}_{\le k}=R_{\le k}^d$. The admissible space consisting of tangential vector fields is
\[
\Pt{\mP}_{\le k}
=\{\xi\in R_{\le k}^d \mid \langle\alpha_i,\xi(x)\rangle=0
\text{ for }x\in F_i,\ 1\le i\le m\}.
\]
Its union $\Pt{\mP}$ over all $k$ is an $R$-module since sums and polynomial multiples preserve tangency. We write
\[
\xi[h]=\langle\nabla h,\xi\rangle
=\sum_{q=1}^d\xi_q\partial_q h,
\qquad \partial_q=\partial/\partial x_q,
\]
where \(\xi_q\in R\) is the \(q\)-th component of \(\xi\).

A polynomial boundary condition extends from each facet to its entire supporting hyperplane.
\begin{lemma}\label{lem:1tangent}
Let $h(x)=\langle\alpha,x\rangle+\ell$ with $\alpha\ne0$, and let $F\subset H=\{h=0\}$ contain a nonempty relatively open subset of $H$. A polynomial $p$ vanishes on $F$ if and only if $p$ is in the principal ideal $(h)$. Consequently,
\[
\xi\in\Pt{\mP}\quad\Longleftrightarrow\quad
\xi[h_i]\in (h_i)\quad(1\le i\le m).
\]
\end{lemma}
\begin{proof}
Choose affine coordinates in which $h$ is one coordinate. The restriction of $p$ to $H$ is a polynomial in $d-1$ variables; if it vanishes on a relatively open set, it is identically zero. Thus the term independent of $h$ vanishes, which is equivalent to divisibility by $h$. Apply this to $p=\xi[h_i]$ on each facet. For $d=1$, the restriction is a constant and the same argument applies.
\end{proof}

\begin{example}\label{ex:triangle}
On the triangle $x\le0$, $y\le0$, $h=x+y+1\ge0$, consider
\[
\xi=(xy,y^2+y)=\underbrace{(xy,y^2)}_{\xi^{(2)}}+
\underbrace{(0,y)}_{\xi^{(1)}}.
\]
The factors $x$ and $y$ give tangency on the coordinate edges, and $\xi[h]=y(x+y+1)=yh$ gives tangency on the third edge. Thus $\xi\in \Pt{\mP}_{\le 2}$. On that edge, however,
\[
\left.\xi^{(2)}[h]\right|_{h=0}=-y,
\qquad
\left.\xi^{(1)}[h]\right|_{h=0}=y.
\]
Neither homogeneous component is tangential there. Thus the admissible space is filtered by degree, but is not graded under the usual degree decomposition. Computing tangential fields separately at each homogeneous degree would miss $\xi$.
\end{example}

\begin{problem}\label{prb:degree}
Given a finite set of samples $\mathcal D=\{(x_s,u_s)\}_{s\in\Obs}\subset\mP\times\R^d$ and a degree bound $k$, minimize
\begin{equation}\label{eq:error}
\error(\xi,\mathcal D)=\sum_{s\in\Obs}\|\xi(x_s)-u_s\|^2
\quad\text{over }\xi\in\Pt{\mP}_{\le k}.
\end{equation}
\end{problem}
A zero minimum gives exact interpolation. 

\section{Basis construction from syzygies}\label{sec:interpolation}
To construct bases at each degree from a common set of module generators, we first homogenize the boundary equations, retaining the cancellations between different degrees seen in \Cref{ex:triangle}. We then describe tangency through polynomial relations, called syzygies. Throughout this section, tuples representing vector fields and syzygies are columns.

\subsection{Homogenization}
Put $S=\R[x_0,x_1,\ldots,x_d]$, write $S_k$ for its homogeneous degree-$k$ part, and set $S_k=\{0\}$ for $k<0$. Define
\[
\hh_0=x_0,\qquad
\hh_i=\ell_i x_0+\langle\alpha_i,x\rangle\ (1\le i\le m),
\qquad Q=\prod_{i=0}^m\hh_i.
\]
The forms are pairwise nonproportional, so $Q$ is squarefree of degree $m+1$. The slice $x_0=1$ recovers the original facet equations. 

The hyperplanes $\hmp=\{\hh_i=0 \mid 0\le i\le m\}$ all pass through the origin; such a finite family is called a central hyperplane arrangement. Its \emph{module of logarithmic derivations} is defined as
\begin{equation}\label{def:logarithmic}
D(Q)=\{\theta\in S^{d+1}\mid\theta[Q]\in(Q)\}.
\end{equation}
The following identity describes $D(Q)$ as the space of polynomial fields tangent to every hyperplane in this family~\cite[Chap.~4]{OrlikTerao1992}:
\begin{equation}\label{prop:Smodule}
D(Q)=\{\theta\in S^{d+1}\mid \theta[\hh_i]\in (\hh_i) \ (0\le i\le m)\}.
\end{equation}
Indeed, for $\theta\in D(Q)$, restrict the product rule for $\theta[Q]$ to $\{\hh_i=0\}$. Only the term $\theta[\hh_i]Q/\hh_i$ remains, and it vanishes there. The restriction of $Q/\hh_i$ is a nonzero polynomial, so $\theta[\hh_i]$ must vanish there as well. Divisibility follows from \Cref{lem:1tangent}. The reverse implication follows easily from the product rule.

Set $D_0(Q)=\{\theta\in D(Q) \mid \theta_0=0\}$. For $f\in R_{\le k}$, let
$\operatorname{hom}_k(f)=x_0^k f(x/x_0)$, interpreted as a polynomial.
\begin{lemma}\label{embedding}
Homogenization and restriction to $x_0=1$ give inverse vector-space isomorphisms
\[
\Pt{\mP}_{\le k}\cong D_0(Q)_k,
\qquad
\xi\longmapsto(0,\operatorname{hom}_k(\xi_1),\ldots,
\operatorname{hom}_k(\xi_d))^{\mathsf T}.
\]
\end{lemma}
\begin{proof}
For $k\ge1$, homogenizing $\xi[h_i]=h_i g_i$, with $g_i\in R_{\le k-1}$, gives
\[
\sum_{q=1}^d(\alpha_i)_q\operatorname{hom}_k(\xi_q)
=\hh_i\operatorname{hom}_{k-1}(g_i).
\]
For $k=0$, divisibility forces $\xi[h_i]=0$. Thus homogenization maps into $D_0(Q)_k$ by \eqref{prop:Smodule}. Restricting these identities to $x_0=1$ gives the inverse, since homogeneous polynomials of degree $k$ are determined by that restriction.
\end{proof}

In the cone
\[
C(\mP)=\{(t,tp):t\ge 0,\ p\in\mP\}\subset\R^{d+1},
\]
$\mP$ is embedded as a slice by $p\mapsto(1,p)$. The homogeneous lift $\widehat\xi$ of a field $\xi$ of degree at most $k$ satisfies
\[
\widehat\xi(t,tp)=t^k(0,\xi(p))^{\mathsf T}.
\]
Thus the lifted vectors are parallel along each ray of the cone, scale by $t^k$, and are tangent to the horizontal slices as well as the cone boundary.

For the triangle and quadratic field $\xi$ in \Cref{ex:triangle}, write $z=x_0$. The lift and cone are
\[
\begin{gathered}
\widehat\xi(z,x,y)=(0,xy,y(y+z))^{\mathsf T},
\\
C(\mP)=\{(z,x,y):z\ge0,\ x\le0,\ y\le0,\ x+y+z\ge0\}.
\end{gathered}
\]
The slices $z=1$ and $z=t$ therefore carry parallel vectors with length ratio $t^2$, as in \Cref{fig:cone-fields}(a).

\subsection{Identification with syzygies}
The partial derivatives of $Q$ generate the Jacobian ideal $J(Q)=(\partial_0Q,\ldots,\partial_dQ)$. We denote the module of relations among this full ordered tuple by
\begin{equation}\label{eq:syz2mod}
\Syz(Q)=\left\{g\in S^{d+1} \mid \sum_{q=0}^d g_q\partial_q Q=0\right\}.
\end{equation}
We grade $\Syz(Q)$ by coefficient degree: $\Syz(Q)_k=\Syz(Q)\cap S_k^{d+1}$. 

\begin{theorem}\label{thm:poly-syz}
For every $k\ge0$, the map
\begin{equation}\label{eq:affine-recovery}
\begin{gathered}
\mathcal T_k:\Syz(Q)_k\longrightarrow\Pt{\mP}_{\le k},\\
(\mathcal T_k g)_q(x)=g_q(1,x)-x_qg_0(1,x),\qquad 1\le q\le d,
\end{gathered}
\end{equation}
is a vector-space isomorphism.
\end{theorem}
\begin{proof}
Let $\theta_E=(x_0,x_1,\ldots,x_d)^{\mathsf T}$ be the Euler field. Euler's identity gives $\theta_E[Q]=(m+1)Q$, and $(\theta_E)_0=x_0$. These identities provide two ways to subtract an Euler multiple:
\[
\pi_0(\theta)=\theta-\frac{\theta_0}{x_0}\theta_E,
\qquad
\pi_{\Syz}(\theta)=\theta-\frac{\theta[Q]}{(m+1)Q}\theta_E.
\]
Both quotients are polynomials by \eqref{prop:Smodule} and \eqref{def:logarithmic}. The first map sets the zeroth component to zero; the second makes the action on $Q$ zero. Each preserves coefficient degree, fixes its target module, and has kernel $S\theta_E$. Thus every field in $D(Q)$ decomposes uniquely into an Euler multiple and a field in either target~\cite[Chap.~4]{OrlikTerao1992}:
\[
D(Q)=S\theta_E\oplus D_0(Q)=S\theta_E\oplus \Syz(Q).
\]
The restrictions $\pi_0|_{\Syz(Q)}$ and $\pi_{\Syz}|_{D_0(Q)}$ are consequently inverse isomorphisms. Composing the former with dehomogenization from \Cref{embedding} gives \eqref{eq:affine-recovery}. Since $x_0\mid g_0$, the correction term has affine degree at most $k$.
\end{proof}

The projection $\pi_0$ removes the transverse component by moving along the cone's radial direction. Since $\theta_E$ is tangent to every hyperplane through the origin, this correction preserves boundary tangency. For the triangle in \Cref{ex:triangle}, take $Q=zxy(x+y+z)$ and set
\[
a=\frac{z+3y}{4},\qquad
g=\widehat\xi-a\theta_E\in \Syz(Q),\qquad
\pi_0(g)=g+a\theta_E=\widehat\xi.
\]
Indeed, $\widehat\xi[Q]=(z+3y)Q$ and $\theta_E[Q]=4Q$. As \Cref{fig:cone-fields}(b) shows, adding $a\theta_E$ brings $g$ into the horizontal slice; restriction to $z=1$ then recovers $\xi$. Simply discarding the $z$-component of $g$ would generally destroy tangency to the triangle's edges.

\begin{figure}[htbp]
\centering
\includegraphics[width=\linewidth]{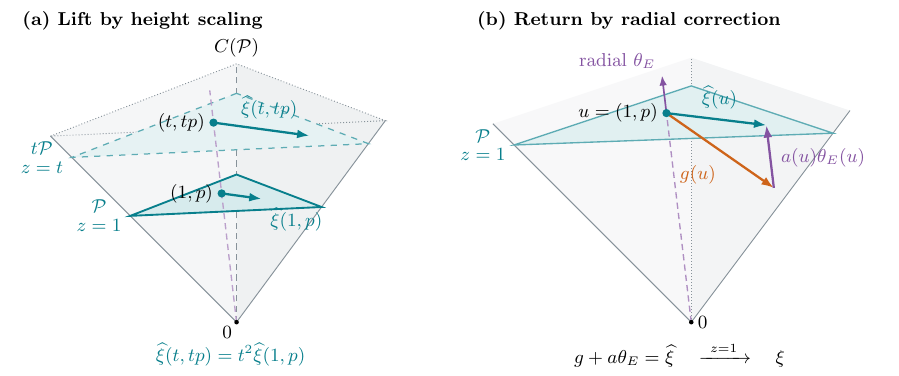}
\caption{Affine and homogeneous fields for the triangle in \Cref{ex:triangle}. (a) The polygon is the slice $z=1$ of its cone, shown truncated. At corresponding points on a ray, the horizontal lifted vectors are parallel and scale by $t^2$, the power prescribed by quadratic homogenization. (b) At a point of the unit-height slice, the three arrows show the vector sum $g+a\theta_E=\widehat\xi$. The correction is along the radial Euler direction, indicated by the dashed ray; the resulting horizontal vector restricts to the original affine field.}
\label{fig:cone-fields}
\end{figure}

\subsection{Generators and dimensions}
Homogeneous generators $s_1,\ldots,s_r$ of $\Syz(Q)$ can be computed by Gr\"obner-basis methods, for example Schreyer's algorithm~\cite{Schreyer1991,Eisenbud}. Here \emph{generate} means spanning with polynomial coefficients, whereas a fitting basis at a fixed degree spans with real coefficients. If $a_j$ is the coefficient degree of $s_j$, then
\begin{equation}\label{eq:degree-candidates}
\Syz(Q)_k=\operatorname{span}_{\R}
\{x_0^{\beta_0}\cdots x_d^{\beta_d}s_j \mid |\beta|=k-a_j\ge0\}.
\end{equation}
Taking independent columns in the polynomial coefficient matrix of these candidates gives a vector-space basis. 

The generators must represent relations in the original Jacobian coordinates, including any zero or redundant entries, as in \eqref{eq:syz2mod}.

The model dimension can be read from the rank of its coefficient matrix, or from a graded free resolution when one is available. If
\[
0\longrightarrow F_L\longrightarrow\cdots\longrightarrow F_0
\longrightarrow \Syz(Q)\longrightarrow0,
\qquad F_i=\bigoplus_jS(-a_{ij}),
\]
then taking degree-$k$ components gives an exact sequence of finite-dimensional vector spaces. The alternating sum of their dimensions yields
\begin{equation}\label{prop:dimension}
\dim\Pt{\mP}_{\le k}
=\sum_{i=0}^L(-1)^i\sum_j\binom{k-a_{ij}+d}{d},
\end{equation}
where $\binom{r}{d}=0$ for $r<d$. Here $S(-a)_k=S_{k-a}$. If shifts are obtained from a resolution of $S/J(Q)$, the corresponding syzygy shifts must first be reduced by $m=\deg\partial_qQ$ to recover coefficient degrees. 

\begin{example}\label{ex:triangle-basis}
For the triangle in \Cref{ex:triangle}, write $z=x_0$ and take $Q=zxy(x+y+z)$. Directly in $D_0(Q)$, with coordinates ordered as $(z,x,y)$, consider
\[
v_1=(0,xy,y(y+z))^{\mathsf T},\quad
v_2=(0,x(x+z),xy)^{\mathsf T},\quad
v_3=(0,-xy,xy)^{\mathsf T}.
\]
These generate $D_0(Q)$ and have the single generating relation
\begin{equation}\label{eq:triangle-relation}
xv_1-yv_2-zv_3=0.
\end{equation}
To check completeness directly, put $H=x+y+z$. Any $(0,f,g)^{\mathsf T}\in D_0(Q)$ has $f=xa$, $g=yb$, and $xa+yb=Hc$. Reducing modulo $(x,y)$ gives $c\in(x,y)$. Write $c=xu+yv$; coprimality of $x,y$ then gives $a=Hu+yw$ and $b=Hv-xw$ for some $w\in S$. Thus
\[
(f,g)=w(xy,-xy)+u(xH,0)+v(0,yH).
\]
These three pairs, obtained from the Koszul relations for $x,y,H$~\cite[Chap.~17]{Eisenbud}, are $-v_3$, $v_2-v_3$, and $v_1+v_3$ after omitting their zero component. This proves generation. For a relation $r_1v_1+r_2v_2+r_3v_3=0$, adding the last two components gives $yr_1+xr_2=0$. Hence $r_1=xq$, $r_2=-yq$, and substitution gives $r_3=-zq$. Thus every relation is a multiple of \eqref{eq:triangle-relation}, and
\[
0\longrightarrow S(-3)\longrightarrow S(-2)^3\longrightarrow D_0(Q)\longrightarrow0,
\qquad
\dim\Pt{\mP}_{\le k}=3\binom{k}{2}-\binom{k-1}{2}.
\]
The dimensions for $k=2,3,4$ are $3,8,15$. Restricting $v_1,v_2,v_3$ to $z=1$ gives the corresponding affine generators.
\end{example}

\section{Fitting and differential constraints}\label{sec:algorithm}
\subsection{Least-squares fitting}
Let $\phi^1,\ldots,\phi^N$ be a basis of $\Pt{\mP}_{\le k}$, obtained from a basis of $\Syz(Q)_k$ via \Cref{thm:poly-syz}. Writing $\xi=\sum_jc_j\phi^j$ reduces \Cref{prb:degree} to
\begin{equation}\label{eq:least-squares}
\min_{c\in\R^N}\|Ac-b\|^2,\qquad
A_{(s,q),j}=\phi_q^j(x_s),\qquad b_{(s,q)}=(u_s)_q.
\end{equation}
\Cref{alg:deg} provides a concrete algorithm to solve the problem.

Exact interpolation holds precisely when $b\in\operatorname{im}A$. Thus $N\ge d|\Obs|$ is necessary for interpolating arbitrary interior velocities.

\begin{algorithm}[htb]
\small
\caption{Fit a tangential polynomial field of degree at most $k$}
\label{alg:deg}
\begin{algorithmic}[1]
\Require Facet equations $h_1,\ldots,h_m$; degree bound $k\ge0$; observations $\mathcal D$.
\State Form $Q=x_0\prod_{i=1}^m\hh_i$ and compute homogeneous generators $s_j$ of $\Syz(Q)$ in the full Jacobian coordinates.
\State Form the degree-$k$ candidates in \eqref{eq:degree-candidates}.
\State Select independent columns of their polynomial coefficient matrix, obtaining a basis $g^1,\ldots,g^N$ of $\Syz(Q)_k$.
\State Set $\phi_q^j(x)=g_q^j(1,x)-x_qg_0^j(1,x)$ for $1\le q\le d$.
\State Assemble $A$ and $b$ as in \eqref{eq:least-squares}, and solve $\min_c\|Ac-b\|^2$.
\State \Return $\xi=\sum_{j=1}^Nc_j\phi^j$ (the zero field if $N=0$).
\end{algorithmic}
\end{algorithm}

The coefficient matrix in step 3 detects polynomial identities independently of the samples; the evaluation matrix $A$ detects dependencies at the chosen sample locations. For example, the triangle's three quadratic generators give nine homogeneous cubic candidates with one relation, so $N=8$.

\subsection{Degree selection and termination}\label{sec:termination}
At a fixed degree, \Cref{thm:poly-syz} and the least-squares solve ensure that \Cref{alg:deg} returns a minimizer over the entire admissible space. To meet a prescribed error tolerance $\epsilon\ge0$, run the algorithm at successive degrees $k=0,1,\ldots$ and stop when the minimum error is at most $\epsilon$. The following result guarantees termination for \emph{boundary-compatible} observations,
where $u_s\in T_{x_s}$ at every observation point with
\[
T_x=\bigcap_{i \mid h_i(x)=0}\alpha_i^\perp,
\quad (T_x=\R^d \text{ when $x$ is in the interior}).
\]

\begin{proposition}\label{prop:interpolation_possible}
Let $x_1,\ldots,x_n\in\mP$ be distinct, with $n\ge1$. There exists $\xi\in\Pt{\mP}$ satisfying $\xi(x_s)=u_s$ for all $s$ if and only if $u_s\in T_{x_s}$. When this condition holds, one can take
\begin{equation}\label{eq:degree-bound}
\deg\xi\le K:=n-1+\max_s|J_s|,
\qquad J_s=\{i \mid h_i(x_s)\ne0\}.
\end{equation}
\end{proposition}
\begin{proof}
Necessity follows from the boundary condition. For sufficiency, choose $a\in\R^d$ such that $\lambda_s=\langle a,x_s\rangle$ are distinct. Set
\[
L_s(x)=\prod_{t\ne s}
\frac{\langle a,x\rangle-\lambda_t}{\lambda_s-\lambda_t},\qquad
W_s(x)=\prod_{i\in J_s}\frac{h_i(x)}{h_i(x_s)},\qquad
\xi(x)=\sum_{s=1}^n L_s(x)W_s(x)u_s.
\]
Since $L_s(x_t)=\delta_{st}$ and $W_s(x_s)=1$, the field interpolates the data. Each summand is tangent to every facet: it vanishes there if $i\in J_s$, and its constant direction $u_s$ is tangent otherwise. Its degree is at most $n-1+|J_s|$.
\end{proof}

Writing
\[
e_k=\min_{\xi\in\Pt{\mP}_{\le k}}\error(\xi,\mathcal D),
\]
the inclusion $\Pt{\mP}_{\le k}\subseteq\Pt{\mP}_{\le k+1}$ gives $e_{k+1}\le e_k$, and \Cref{prop:interpolation_possible} gives $e_K=0$ for boundary-compatible data. Thus, for every $\epsilon\ge0$, the degree search stops at the smallest $k$ with $e_k\le\epsilon$, no later than $K$. 

\subsection{Additional linear conditions}\label{sec:differential}
For a linear differential operator $\mathcal L$, form the polynomial coefficient matrix $C_{\mathcal L}$ whose $j$th column represents $\mathcal L\phi^j$. If the columns of $Z$ form a basis of $\ker C_{\mathcal L}$, then, writing $\phi=(\phi^1,\ldots,\phi^N)$, the columns of $\phi Z$ form a basis of the fields satisfying $\mathcal L\xi=0$ identically, and fitting reduces to $\min_a\|AZa-b\|^2$. Relevant choices are
\[
\begin{aligned}
\Dt{\mP}_{\le k}&=\{\xi\in\Pt{\mP}_{\le k} \mid \operatorname{div}\xi=0\},\\
\Rt{\mP}_{\le k}&=\{\xi\in\Pt{\mP}_{\le k} \mid \partial_p\xi_q=\partial_q\xi_p\},\\
\Ht{\mP}_{\le k}&=\{\xi\in\Pt{\mP}_{\le k} \mid \Delta\xi=0\},
\end{aligned}
\]
where the Laplacian acts componentwise. These are vector spaces, generally not $R$-submodules. On a bounded polytope, a field in $\Dt{\mP}$ generates a volume-preserving flow that remains in the domain.

The interpolation existence extends to divergence-free velocity fitting in dimension at least two.
\begin{proposition}\label{prop:divfree-interpolation}
Let $d\ge2$ and let $x_1,\ldots,x_n\in\mP$ be distinct, with $n\ge1$. If $u_s\in T_{x_s}$ for every $s$, there exists $\xi\in\Dt{\mP}_{\le 2K}$ satisfying $\xi(x_s)=u_s$, where $K$ is given by \eqref{eq:degree-bound}.
\end{proposition}
\begin{proof}
Use $L_s,W_s$ from the proof of \Cref{prop:interpolation_possible}, and put $b_s=(L_sW_s)^2$ and $y_s=x-x_s$. For a scalar polynomial $b$, define
\begin{equation}\label{eq:divfree-lift}
\mathcal R_{s,u}(b)
=bu+\frac{(y_s\cdot\nabla b)u-(u\cdot\nabla b)y_s}{d-1}.
\end{equation}
For constant $u$, direct differentiation gives $\operatorname{div}\mathcal R_{s,u}(b)=0$: the divergence of the numerator is $-(d-1)u\cdot\nabla b$. Set $\xi_s=\mathcal R_{s,u_s}(b_s)$. At $x_s$, we have $b_s=1$ and $y_s=0$, so $\xi_s(x_s)=u_s$. At every $x_t$ with $t\ne s$, both $b_s$ and $\nabla b_s$ vanish, so $\xi_s(x_t)=0$.

If $i\in J_s$, then $h_i^2\mid b_s$, and $\xi_s$ vanishes on $H_i$. Otherwise, $\alpha_i\cdot u_s=0$ and $\alpha_i\cdot y_s=0$ on $H_i$, so \eqref{eq:divfree-lift} is tangent there as well. Finally, $\deg\xi_s\le\deg b_s\le2(n-1+|J_s|)$. Summing the $\xi_s$ proves the assertion.
\end{proof}

\begin{remark}\label{rem:harmonic-obstruction}
An arbitrary additional differential constraint need not ensure the interpolation existence. For example, on $\mP=[0,1]^2$, the space $\Ht{\mP}$ consists only of the zero field. Indeed, a harmonic polynomial $p(x,y)$ vanishing on $x=0$ and $x=1$ must be zero: its highest nonzero coefficient as a polynomial in $y$ would be affine in $x$ by harmonicity, but would vanish at both endpoints.
\end{remark}

In the planar divergence-free case, a hamiltonian construction gives a smaller basis directly.
\begin{proposition}\label{prop:streamfunction}
Let $\mP\subset\R^2$ be a convex polygon with $m$ edges, and set $B=\prod_{i=1}^m h_i$. With $\nabla^\perp\psi=(\partial_2\psi,-\partial_1\psi)$, the map
\[
R_{\le k+1-m}\longrightarrow\Dt{\mP}_{\le k},\qquad
p\longmapsto\nabla^\perp(Bp)
\]
is an isomorphism. In particular,
\[
\dim\Dt{\mP}_{\le k}=\binom{k+3-m}{2},
\]
with the zero convention used in \eqref{prop:dimension}.
\end{proposition}
\begin{proof}
Every divergence-free polynomial field of degree at most $k$ has a polynomial stream function $\psi$ of degree at most $k+1$, obtained by integrating $(-\xi_2,\xi_1)$. Tangency makes $\psi$ constant on each edge. Adjacent edges meet, so these constants agree around the boundary. After subtracting the common constant, each $h_i$ divides $\psi$ by \Cref{lem:1tangent}, hence $\psi=Bp$. Conversely, $\nabla^\perp(Bp)$ is divergence-free and tangent. The map is injective because $Bp$ can be constant only when $p=0$.
\end{proof}

Derivative observations fit into the same framework by replacing evaluation with the observed linear functional. For example, planar vorticity uses
\[
\operatorname{curl}\xi=\partial_1\xi_2-\partial_2\xi_1,
\qquad W_{s,j}=(\operatorname{curl}\phi^j)(x_s),
\]
and the objective is $\|Wc-\omega\|^2$. Vorticity alone does not generally determine a tangential field. When $m\ge1$, the field $\nabla(B^2)$ is nonzero, curl-free, and zero on the boundary, with degree $2m-1$.

Requiring zero divergence removes this ambiguity when vorticity is known throughout a bounded convex domain. Indeed, the difference of two such fields is curl-free, hence is $\nabla\Phi$ for a polynomial potential, and satisfies $\Delta\Phi=0$ and $\partial_n\Phi=0$. Green's identity yields
\[
\int_{\mP}|\nabla\Phi|^2
=\int_{\partial\mP}\Phi\,\partial_n\Phi-\int_{\mP}\Phi\,\Delta\Phi=0.
\]
Hence $\nabla\Phi=0$, and the two fields coincide.

\section{Piecewise polynomial fields}\label{sec:piecewise}
Let $\Delta=\{\mP_1,\ldots,\mP_C\}$ be a finite face-to-face decomposition of a polyhedral domain $\mQ$: the cells are full-dimensional convex polyhedra with disjoint interiors, and every nonempty intersection is a common facet. 

For an integer $r\ge0$, let $C^r(\mQ;\R^d)$ denote fields that are $C^r$ in the interior, with all derivatives through order $r$ continuous up to the boundary. We consider the space of $C^r$ piecewise polynomial fields tangent to the exterior boundary,
\[
\PPt{r}{\Delta}_{\le k}
=\left\{\xi\in C^r(\mQ;\R^d)\;\middle|\;
\begin{array}{l}
\xi|_{\mP_c}\in\Poly{\mP_c}_{\le k}\quad\text{for every }c,\\
\xi\text{ is tangent to }\partial\mQ
\end{array}\right\}.
\]
Here $r=0$ means continuity. Tangency is imposed on $\partial\mQ$, including the boundaries of holes, but not on interior facets. We characterize this space by divisibility conditions on the local polynomial fields.

For each cell $\mP_c$, let $\Ext(\mP_c)$ index its facets contained in $\partial\mQ$, and define
\[
\Pext{\mP_c}_{\le k}=\{\xi\in\Poly{\mP_c}_{\le k}\mid\xi[h_i]\in(h_i)\text{ for }i\in\Ext(\mP_c)\}.
\]
The proof of \Cref{thm:poly-syz} applies to these selected facet equations, giving
\[
\Pext{\mP_c}_{\le k}\cong\Syz(Q_c)_k,
\qquad Q_c=x_0\prod_{i\in\Ext(\mP_c)}\hh_i.
\]
If there are no exterior facets, this local space is $\Poly{\mP_c}_{\le k}$.

For each pair of cells sharing a facet, let $h_{cc'}=0$ be its supporting hyperplane.
\begin{proposition}\label{prop:gluedlinear}
Local fields $\xi_c\in\Pext{\mP_c}_{\le k}$ are the restrictions of a field in $\PPt{r}{\Delta}_{\le k}$ if and only if
\begin{equation}\label{eq:piecewise-matching}
h_{cc'}^{r+1}\mid\xi_c-\xi_{c'}
\qquad\text{componentwise for every shared facet.}
\end{equation}
\end{proposition}
\begin{proof}
For a facet with equation $h=0$, choose affine coordinates with $y_1=h$. Expanding a polynomial in powers of $y_1$ shows that its derivatives through order $r$ vanish on the facet if and only if it is divisible by $h^{r+1}$. Thus \eqref{eq:piecewise-matching} is equivalent to agreement of these derivatives on shared facets. 
\end{proof}

The divisibility conditions \eqref{eq:piecewise-matching} are the usual polynomial matching conditions in spline constructions~\cite{Billera1989}. In local bases $\Pext{\mP_c}_{\le k}$, they give linear equations for the coefficients and hence a basis of $\PPt{r}{\Delta}_{\le k}$. For a basis $\Phi^1,\ldots,\Phi^D$, the evaluation matrix is simply $A_{(s,q),j}=\Phi^j_q(x_s)$, and the fitting procedure of \Cref{sec:algorithm} applies. Differential constraints may be imposed on the local spaces before matching.

The next result guarantees interpolation at a sufficiently high degree, including observations on interfaces.
\begin{proposition}\label{prop:piecewise-interpolation}
Fix $r\ge0$ and let $x_1,\ldots,x_n\in\mQ$ be distinct, with $n\ge1$. There exists a finite $k$ and a field $\xi\in\PPt{r}{\Delta}_{\le k}$ with $\xi(x_s)=u_s$ for every $s$ if and only if $u_s$ is tangent to every exterior facet containing $x_s$. If $d\ge2$, the field may additionally be chosen divergence-free on every cell.
\end{proposition}
\begin{proof}
Fix $s$ and let $\mathcal C_s=\{c\mid x_s\in\mP_c\}$. Let $\mathcal E_s$ be the exterior facets of these cells not containing $x_s$, and let $\mathcal F_s$ be the shared facets between a cell in $\mathcal C_s$ and a cell outside it. Their facet equations are nonzero at $x_s$, since each is a facet of a cell containing $x_s$ but does not itself contain $x_s$. Using the Lagrange polynomials $L_s$ from the proof of \Cref{prop:interpolation_possible}, set
\[
b_s(x)=L_s(x)\prod_{E\in\mathcal E_s}\frac{h_E(x)}{h_E(x_s)}
\prod_{F\in\mathcal F_s}\left(\frac{h_F(x)}{h_F(x_s)}\right)^{r+1},
\qquad
\eta_{s,c}=\begin{cases}
b_su_s,&c\in\mathcal C_s,\\
0,&c\notin\mathcal C_s.
\end{cases}
\]
The factors $h_E$ and the direction $u_s$ ensure exterior tangency, while the factors $h_F^{r+1}$ ensure matching by \Cref{prop:gluedlinear}. Since $\eta_s(x_t)=\delta_{st}u_s$, the sum $\sum_s\eta_s$ is the required interpolant.

For $d\ge2$, replace $b_su_s$ by $\mathcal R_{s,u_s}(b_s^2)$ from \eqref{eq:divfree-lift}. As in \Cref{prop:divfree-interpolation}, this preserves exterior tangency and the observation values and gives zero divergence. The components remain divisible by $h_F^{2r+1}$ on each $F\in\mathcal F_s$, since the lift differentiates $b_s^2$ at most once. As $2r+1\ge r+1$, matching is preserved as well.
\end{proof}

\section{Examples}\label{sec:examples}

The examples illustrate five aspects of the construction. \Cref{sec:ex-triangle} relates exact interpolation to the degree and the rank of the evaluation matrix. \Cref{sec:ex-differential} shows how differential constraints select different interpolants of the same data. \Cref{sec:ex-taylor-green} uses Taylor--Green fields to examine the effects of boundary tangency, incompressibility, and the choice of observed quantity. \Cref{sec:ex-stommel} applies the method to a model of wind-driven ocean circulation with a western boundary current. Finally, \Cref{sec:ex-piecewise} examines continuity and higher-order matching on non-convex domains. All computations were carried out with the accompanying code.

We use two measures to assess reconstruction accuracy and boundary compatibility. For a known reference field $u$, we report the relative error on a regular grid $\{x_t\}_{t\in\mathcal T}$ covering $\mP$,
\begin{equation}\label{eq:relerr}
\relerr(\xi)=
\Bigl(\sum_{t\in\mathcal T}\|\xi(x_t)-u(x_t)\|^2\Bigr)^{1/2}
\Big/
\Bigl(\sum_{t\in\mathcal T}\|u(x_t)\|^2\Bigr)^{1/2},
\end{equation}
which is independent of the scale of $u$. The boundary leakage
\begin{equation}\label{eq:leak}
\leak(\xi)=\max_{x\in S}\,|\langle\xi(x),n(x)\rangle|
\end{equation}
is the largest normal component over a finite sample $S\subset\partial\mP$ of boundary points, $n$ being the unit normal of the facet sampled. It measures the maximum sampled violation of the no-penetration condition. For $\xi\in\Pt{\mP}$ the normal component vanishes identically; nonzero computed values arise from floating-point evaluation.

\subsection{Degree and exact interpolation on a triangle}\label{sec:ex-triangle}
This example distinguishes exact interpolation from least-squares approximation by comparing the ranks of the evaluation and augmented matrices.
On the triangle $x_1\le0$, $x_2\le0$, $x_1+x_2+1\ge0$, prescribe
\[
\begin{array}{c|c|c}
s&x_s&u_s\\\hline
1&(-1/5,-1/5)&(1,1/2)\\
2&(-1/2,-1/5)&(0,1)\\
3&(-1/5,-1/2)&(0,0)\\
4&(-4/5,-1/10)&(2,2)\\
5&(-1/10,-4/5)&(0,0)\\
6&(-2/5,-2/5)&(1/2,-2)
\end{array}
\]
The zero velocities prescribe stationary points. By \Cref{ex:triangle-basis}, the model dimensions are $8$ at degree $3$ and $15$ at degree $4$. Exact evaluation-matrix ranks for the first $n$ observations are
\[
\begin{array}{c|c|c|c|c}
k&n&\dim\Pt{\mP}_{\le k}&\operatorname{rank}A&\operatorname{rank}[A\ b]\\\hline
3&3&8&6&6\\
3&6&8&8&9\\
4&4&15&8&8\\
4&6&15&12&12
\end{array}
\]
Thus the first three observations admit cubic interpolation, all six do not, and quartics interpolate all six. The corresponding fits are shown in \Cref{fig:triangle-interpolation}.
\begin{figure}[htbp]
\centering
\includegraphics[width=\linewidth]{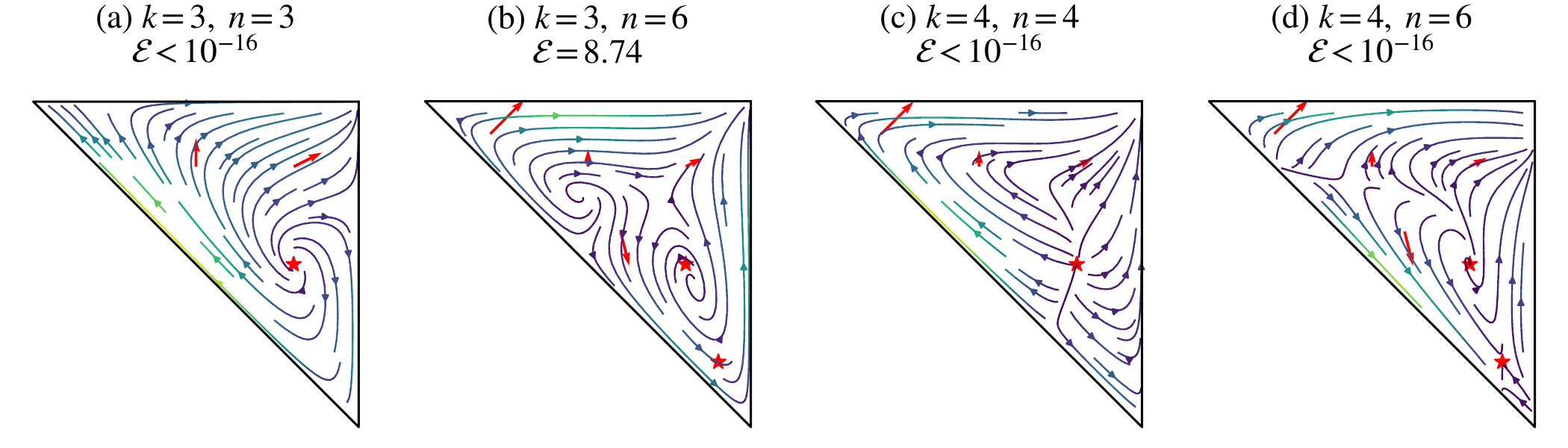}
\caption{Tangential interpolation on a triangle. Panels use $(k,n)=(3,3),(3,6),(4,4),(4,6)$. Red arrows show prescribed velocities and asterisks mark zero velocities. The displayed $\error$ is the sum of squared errors \eqref{eq:error}.}
\label{fig:triangle-interpolation}
\end{figure}

\subsection{Selection of interpolants by differential constraints}\label{sec:ex-differential}
Here the degree and observations are fixed, while the differential constraint varies. The purpose is to compare the resulting admissible spaces and the uniqueness of interpolation within each space.
On $\mP=\{(x,y) \mid x\ge0,\ y\ge0,\ x+y\le1\}$, prescribe
\[
\begin{array}{c|c}
x_s&u_s\\\hline
(0.25,0.50)&(1,0)\\
(0.50,0.33)&(0.5,1)\\
(0.33,0.10)&(1,0)\\
(0.16,0.10)&(0,0)
\end{array}
\]
At degree $8$, the tangential space has dimension $63$. Imposing curl-free, divergence-free, or componentwise harmonic conditions gives dimensions $28$, $28$, and $8$, respectively. The divergence-free count also follows from \Cref{prop:streamfunction}. In each space the evaluation matrix for these data has rank $8$, so all three admit exact interpolation. The harmonic interpolant is unique within its model space; the other two are not. The three fields agree at the four sample points but have different behavior away from them. 
The three interpolants are shown in \Cref{fig:divfree}.
\begin{figure}[htbp]
\centering
\includegraphics[width=\linewidth]{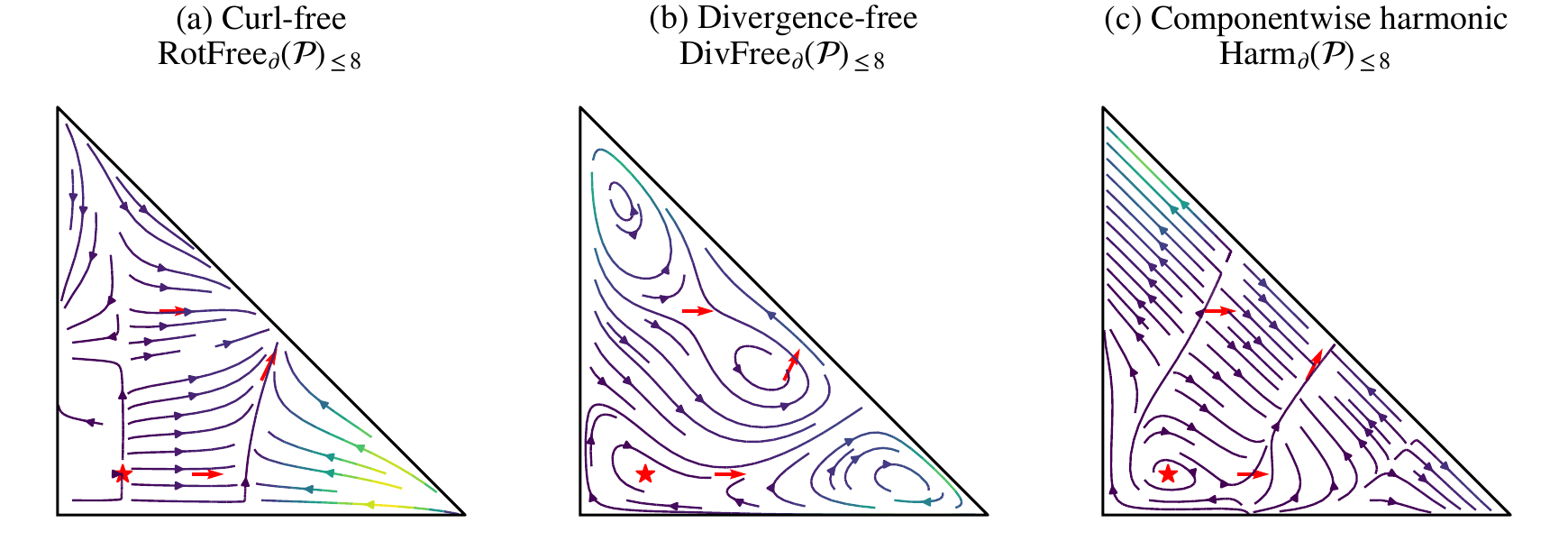}
\caption{Degree-$8$ tangential interpolation with curl-free (left), divergence-free (middle), and componentwise harmonic (right) constraints. All three spaces interpolate the four prescribed velocities.}
\label{fig:divfree}
\end{figure}

\subsection{Taylor--Green fields: boundary conditions and observations}\label{sec:ex-taylor-green}
The Taylor--Green fields~\cite{TaylorGreen1937} provide analytic reference flows for studying how geometric and differential constraints affect reconstruction from sparse, noisy observations. The fields used below are incompressible and satisfy free-slip conditions on the edges of a square or the faces of a cube: the normal velocity and tangential viscous stress vanish, while tangential motion is permitted. The space $\Pt{\mP}$ enforces the no-penetration part of this condition. $\Dt{\mP}$ additionally enforces incompressibility. We first examine velocity reconstruction on the square, then recovery of the same planar field from vorticity observations, and finally velocity reconstruction on the cube.

\subsubsection{Velocity reconstruction in two dimensions}\label{sec:ex-tg-velocity}
On the unit square $\mP=[0,1]^2$, consider the planar Taylor--Green field
\begin{equation}\label{eq:tg2d}
u=(\sin\pi x\cos\pi y,\ -\cos\pi x\sin\pi y),
\qquad
\operatorname{curl}u=2\pi\sin\pi x\sin\pi y.
\end{equation}
Up to a time-dependent amplitude, this field is an exact solution of the incompressible Navier--Stokes equations. It is divergence-free, satisfies $u\cdot n=0$ on all four walls, and has $\partial_yu_1+\partial_xu_2=0$.

We sample the velocity at $50$ interior points and add Gaussian noise of standard deviation $5\%$ of the sampled velocity spread. The fits give
\[
\begin{array}{l|c|c|c|c}
\text{model}&k&\dim&\relerr&\leak\\\hline
R_{\le k}^2&4&30&25.5\%&1.5\\
\Pt{\mP}_{\le k}&6&30&9.1\%&3\cdot10^{-14}\\
\Dt{\mP}_{\le k}&7&15&5.2\%&6\cdot10^{-14}
\end{array}
\]
The unconstrained and tangential models have the same dimension. Imposing tangency reduces the relative error from $25.5\%$ to $9.1\%$, and imposing incompressibility further reduces it to $5.2\%$ with half as many parameters. The unconstrained leakage is $1.5$, compared with the reference root-mean-square speed of $0.71$.
\Cref{fig:tg2d-velocity} shows the reference and reconstructed fields.

To distinguish approximation accuracy from sensitivity to noise, we also vary the degree of the divergence-free tangential model:
\[
\begin{array}{c|c|c|c|c}
k&5&7&9&11\\\hline
\dim\Dt{\mP}_{\le k}&6&15&28&45\\
\text{noise-free }\relerr&0.90\%&0.035\%&0.0012\%&<10^{-5}\\
5\%\text{ noise }\relerr&1.9\%&5.2\%&8.0\%&16.8\%
\end{array}
\]
The noise-free errors decrease rapidly with degree. For the noisy data, the smallest error among the tested degrees occurs at $k=5$; degrees $3$ and $4$ both give $14.5\%$. The subsequent increase is consistent with greater sensitivity to noise as the dimension grows at a fixed sample size. Exact boundary tangency is preserved throughout the degree sweep.
\begin{figure}[htbp]
\centering
\includegraphics[width=\linewidth]{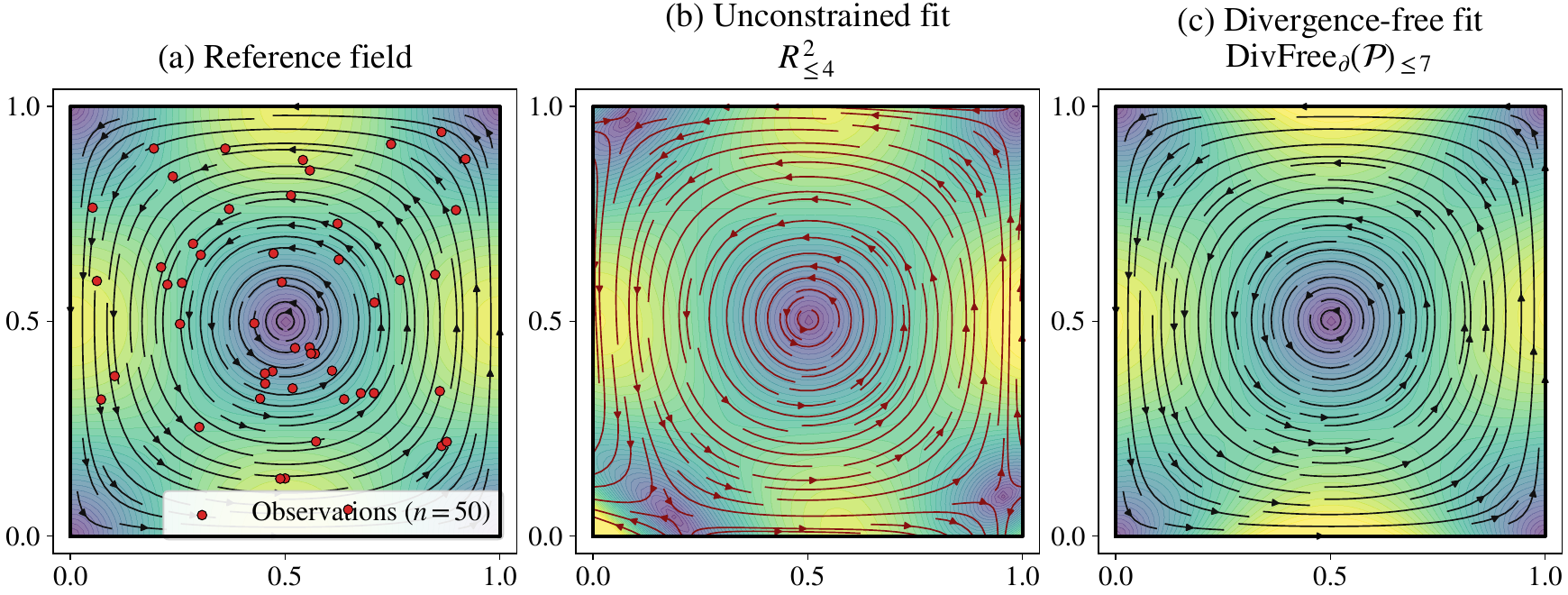}
\caption{Velocity reconstruction of the planar Taylor--Green field from $50$ noisy samples. (a) Reference field and sample locations; (b) unconstrained degree-$4$ fit of dimension $30$; (c) divergence-free tangential degree-$7$ fit of dimension $15$.}
\label{fig:tg2d-velocity}
\end{figure}

\subsubsection{Reconstruction from vorticity observations}\label{sec:ex-vorticity}
We now use the same planar field \eqref{eq:tg2d} to examine the role of the observed quantity. As described in \Cref{sec:differential}, velocity evaluation can be replaced by any linear functional of the field. We observe scalar vorticity at $50$ interior points, add Gaussian noise of standard deviation $5\%$ of the sampled vorticity spread, and minimize $\|Wc-\omega\|^2$. The question is whether a fit to these derivative observations also recovers the velocity.

The vorticity objective is unchanged by adding a field in $\ker C_{\mathrm{curl}}$, the curl-free subspace of the tangential model. On the square, its dimension is given as follows:
\[
\begin{array}{c|c|c|c}
k&\dim\Pt{\mP}_{\le k}&\dim\ker C_{\mathrm{curl}}&\dim\Dt{\mP}_{\le k}\\\hline
3&6&4&1\\
5&20&9&6\\
7&42&18&15\\
9&72&31&28
\end{array}
\]

At degree $6$, the tangential and divergence-free tangential spaces have dimensions $30$ and $10$, respectively. The velocity fit from \Cref{sec:ex-tg-velocity} provides the comparison in the first row below.
\[
\begin{array}{l|c|c|c}
\text{objective and model}&\dim&\text{velocity }\relerr&\text{vorticity }\relerr\\\hline
\text{velocity samples, }\Pt{\mP}_{\le6}&30&9.1\%&16.0\%\\
\text{vorticity samples, }\Pt{\mP}_{\le6}&30&184\%&6.4\%\\
\text{vorticity samples, }\Dt{\mP}_{\le6}&10&1.0\%&3.1\%
\end{array}
\]
The vorticity fit in the full tangential space has a smaller vorticity error than the velocity fit, but its velocity error is $184\%$ (\Cref{fig:vorticity}(b)). Across $3\le k\le9$ its velocity error never falls below $140\%$ and reaches $2\cdot10^4\%$ at $k=9$, while its vorticity error is $3.5\%$ at $k=5$. Thus accurate approximation of the vorticity need not imply accurate recovery of the velocity in the full tangential space.

Requiring zero divergence eliminates the curl-free ambiguity, as shown at the end of \Cref{sec:differential} for vorticity known throughout the domain. The divergence-free tangential fit has a velocity error of $1.0\%$; see \Cref{fig:vorticity}(c).
\begin{figure}[htbp]
\centering
\includegraphics[width=0.92\linewidth]{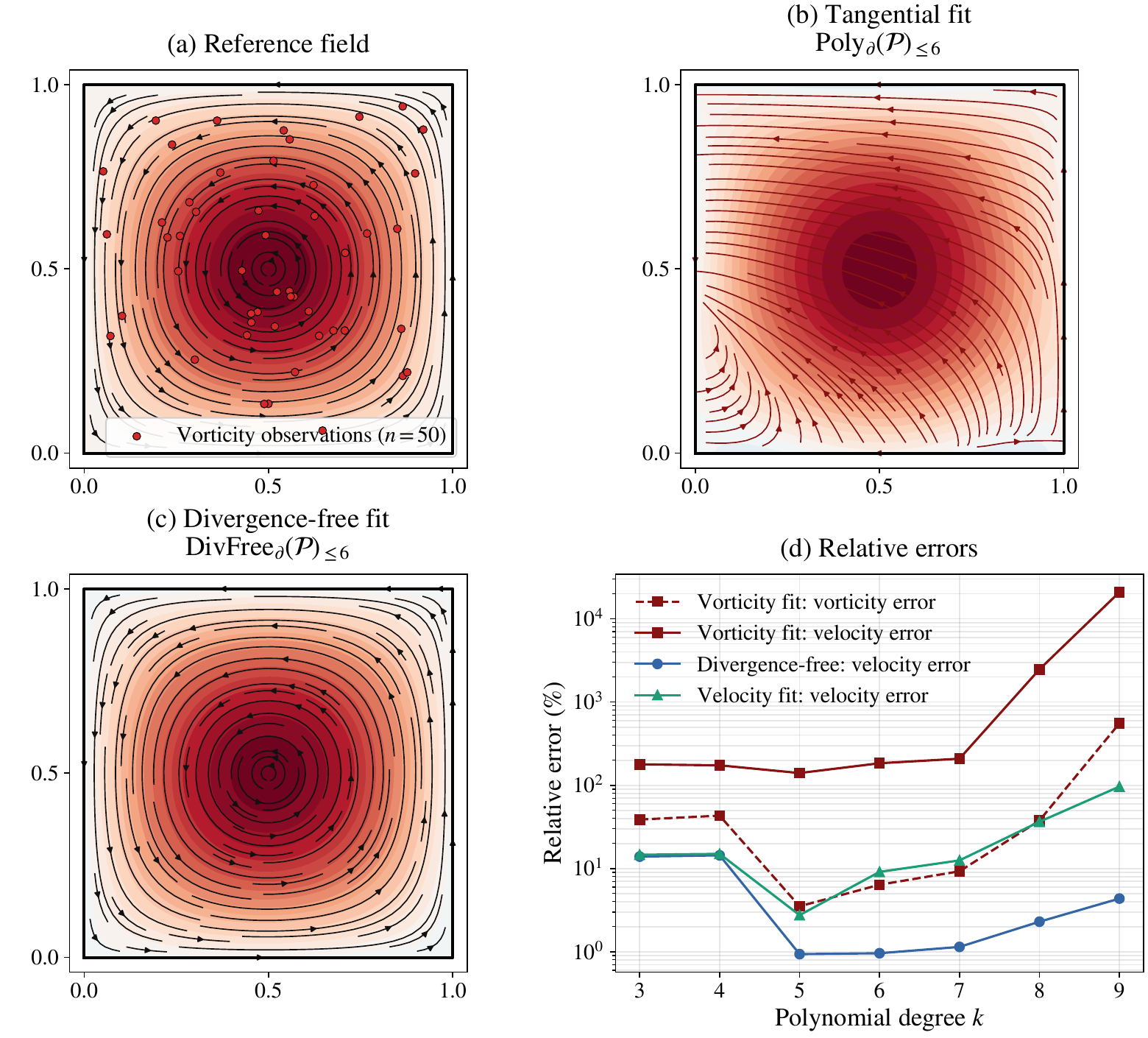}
\caption{Fitting the flow \eqref{eq:tg2d} from $50$ noisy vorticity samples. Color indicates vorticity, and lines are streamlines. (a) the reference and the probe locations; (b) the degree-$6$ tangential fit, which has a small vorticity error but a large velocity error; (c) the degree-$6$ divergence-free tangential fit, of dimension $10$; (d) relative errors \eqref{eq:relerr} against degree.}
\label{fig:vorticity}
\end{figure}

\subsubsection{Velocity reconstruction in three dimensions}\label{sec:ex-3d}
We next examine velocity reconstruction in three dimensions. On $\mP=[0,1]^3$ the six facets fall into three parallel pairs sharing a normal direction, so \Cref{lem:1tangent} reads
\[
x_q(x_q-1)\mid\xi_q\qquad(1\le q\le3),
\]
whence $\dim\Pt{\mP}_{\le k}=3\binom{k+1}{3}$, against $3\binom{k+3}{3}$ for the unconstrained model. Imposing $\operatorname{div}\xi=0$ as in \Cref{sec:differential} leaves
\[
\begin{array}{c|c|c|c|c|c|c|c}
k&2&3&4&5&6&7&8\\\hline
\dim R_{\le k}^3&30&60&105&168&252&360&495\\
\dim\Pt{\mP}_{\le k}&3&12&30&60&105&168&252\\
\dim\Dt{\mP}_{\le k}&0&3&11&26&50&85&133
\end{array}
\]

We use the three-dimensional Taylor--Green field
\[
u=\bigl(
\sin\pi x_1\cos\pi x_2\cos\pi x_3,\;
\cos\pi x_1\sin\pi x_2\cos\pi x_3,\;
-2\cos\pi x_1\cos\pi x_2\sin\pi x_3
\bigr).
\]
This field is divergence-free and satisfies $u\cdot n=0$ on all six faces. Its tangential velocity need not vanish there, and its tangential shear stresses vanish on each face. We sample it at $80$ interior points and add Gaussian noise of standard deviation $5\%$ of the sampled velocity spread. Fitting the noisy data gives
\[
\begin{array}{l|c|c|c|c}
\text{model}&k&\dim&\relerr&\leak\\\hline
R_{\le k}^3&3&60&77.0\%&2.6\\
R_{\le k}^3&4&105&32.8\%&1.3\\
\Pt{\mP}_{\le k}&5&60&34.6\%&2\cdot10^{-14}\\
\Dt{\mP}_{\le k}&6&50&6.8\%&1\cdot10^{-13}
\end{array}
\]
The first and third rows have the same dimension. In this comparison, imposing tangency reduces the relative error from $77.0\%$ to $34.6\%$ and enforces zero normal velocity identically. The divergence-free tangential model has fewer parameters and yields the smallest error, $6.8\%$. Increasing the unconstrained degree from $3$ to $4$ also reduces the error, but requires $105$ parameters and leaves nonzero boundary leakage.

Boundary leakage also affects the computed trajectories. Of $100$ streamlines seeded in a thin shell inside the walls and integrated, $56$ leave the cube under the unconstrained degree-$3$ fit, and none leave under either constrained fit; see \Cref{fig:tg3d}(b). Using $800$ noise-free samples, the divergence-free model gives the following decrease in error with degree:
\[
\begin{array}{c|c|c|c|c}
k&4&6&8&10\\\hline
\dim\Dt{\mP}_{\le k}&11&50&133&276\\
\relerr&21.5\%&2.7\%&0.22\%&0.014\%
\end{array}
\]
\begin{figure}[htbp]
\centering
\includegraphics[width=\linewidth]{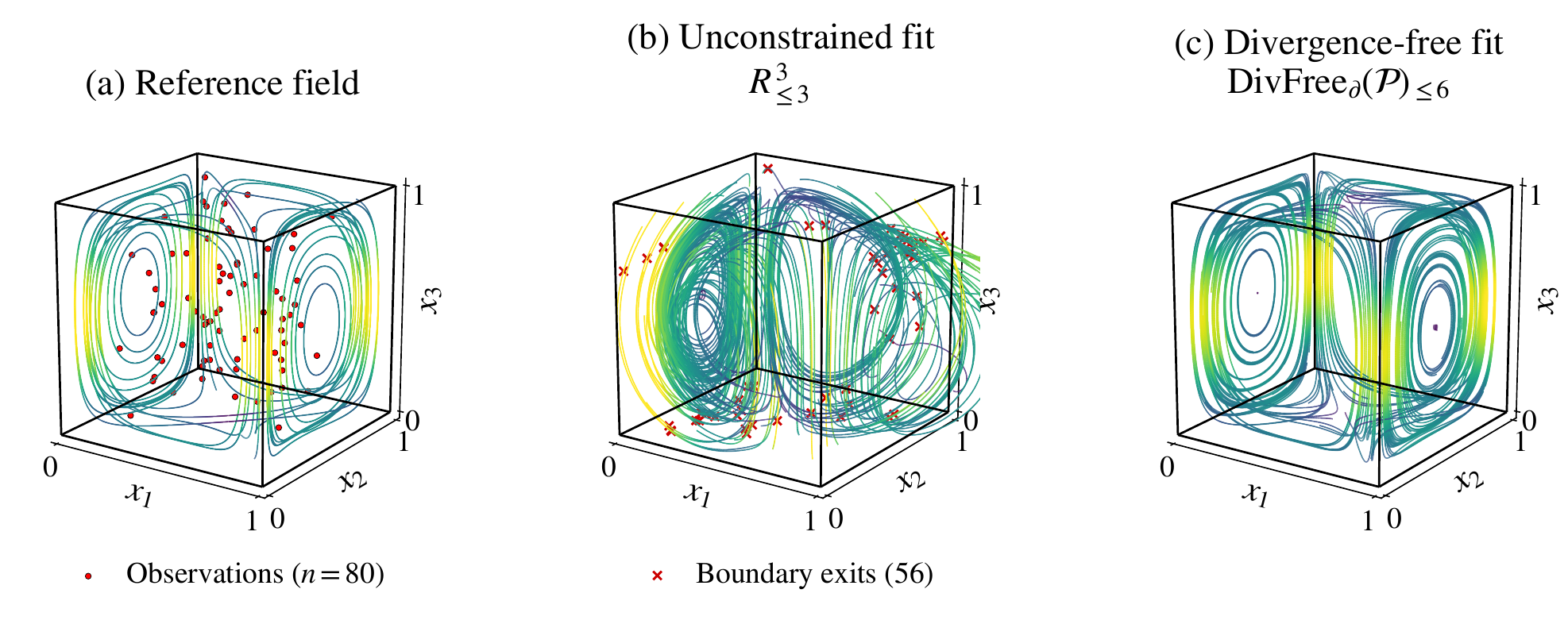}
\caption{Reconstruction of the three-dimensional Taylor--Green vortex in $[0,1]^3$ from $80$ noisy probes. Streamlines are seeded in a shell just inside the walls and colored by speed: (a) the reference flow; (b) the unconstrained degree-$3$ fit; (c) the divergence-free tangential degree-$6$ fit. Crosses mark the points at which a streamline leaves the container.}
\label{fig:tg3d}
\end{figure}

\FloatBarrier
\subsection{Stommel's model of wind-driven ocean circulation}\label{sec:ex-stommel}
Stommel's model~\cite{Stommel1948} describes steady wind-driven circulation in a homogeneous rectangular ocean basin, with linear bottom friction and a Coriolis parameter that varies with latitude. The balance between friction and the Coriolis $\beta$ effect produces a narrow western boundary current, providing a model for the western intensification observed in ocean gyres. In this setting, incompressibility and impermeable coastlines are part of the physical model. The reconstruction problem is to recover the basin circulation and its concentrated coastal current from sparse interior velocity observations while preserving these constraints.

For the numerical example on $\mP=[0,1]^2$, we use the reduced western-boundary-layer profile
\begin{equation}\label{eq:stommel-profile}
\psi(x,y)=\sin(\pi y)\Bigl[1-x-\frac{e^{-x/\epsilon}-e^{-1/\epsilon}}{1-e^{-1/\epsilon}}\Bigr],
\qquad u=\nabla^\perp\psi,
\end{equation}
with $\epsilon=0.08$. This profile satisfies
\[
\epsilon\partial_x^2\psi+\partial_x\psi=-\sin(\pi y),
\qquad \psi|_{\partial\mP}=0.
\]
It is the reduced form obtained by omitting the meridional friction term $\epsilon\partial_y^2\psi$ from the nondimensional Stommel equation. Thus the experiment uses a boundary-layer approximation to the basin model. The associated velocity is exactly divergence-free and tangent to the coast, with a current concentrated near $x=0$.  Here the coastal condition is no penetration.

From $60$ interior velocity observations with Gaussian noise of standard deviation $5\%$ of the sampled velocity spread, we obtain
\[
\begin{array}{l|c|c|c|c}
\text{model}&k&\dim&\relerr&\leak\\\hline
R_{\le k}^2&5&42&27.6\%&5.9\\
\Pt{\mP}_{\le k}&6&30&15.9\%&2\cdot10^{-13}\\
\Dt{\mP}_{\le k}&9&28&5.1\%&1\cdot10^{-11}
\end{array}
\]
The tangential models reduce the error while using fewer parameters than the unconstrained model. The divergence-free tangential fit has the smallest error, $5.1\%$, and resolves the western boundary current shown in \Cref{fig:stommel}. 

For the divergence-free tangential model, degrees $k=7,9,11,13$ give relative errors of $8.7\%,5.1\%,4.8\%$, and $29.2\%$, respectively. Among these degrees, the smallest error occurs at $k=11$, reflecting the higher resolution required by the boundary layer; increasing the degree further amplifies the effect of noise.
\begin{figure}[htbp]
\centering
\includegraphics[width=0.88\linewidth]{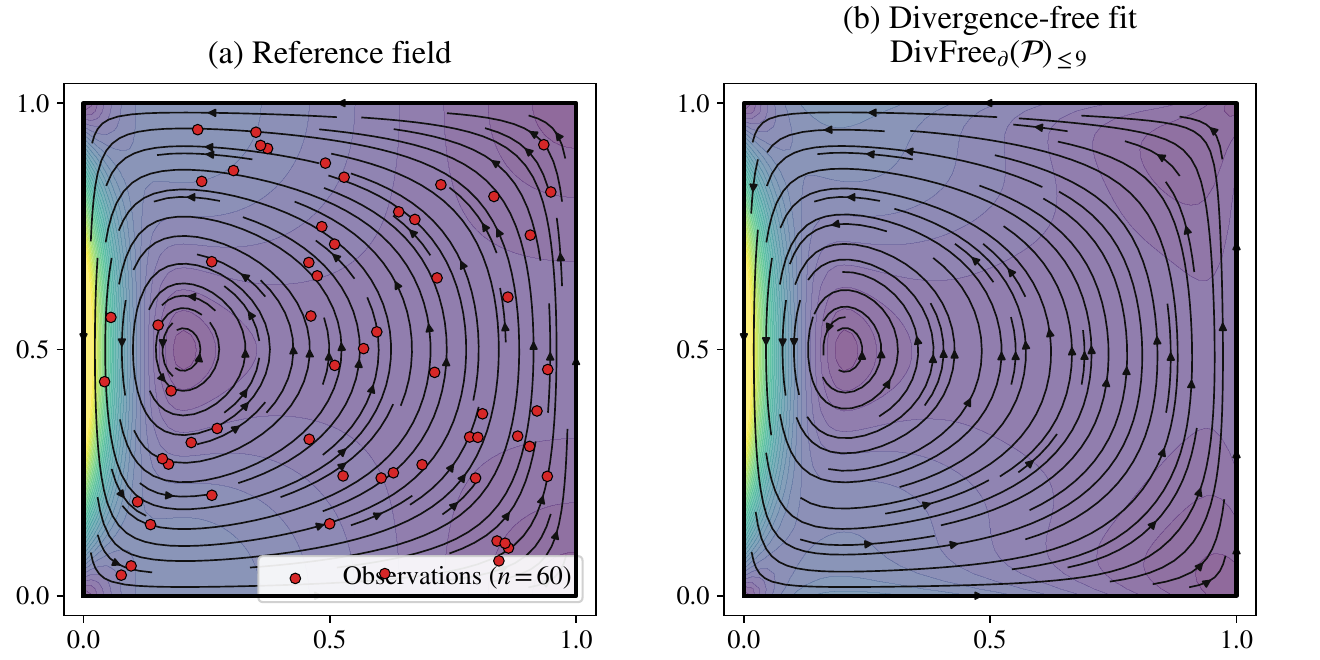}
\caption{Stommel-type basin circulation \eqref{eq:stommel-profile} with $\epsilon=0.08$. Left: reference field and $60$ observation points. Right: the divergence-free tangential degree-$9$ reconstruction, of dimension $28$, from noisy velocity samples. Both fields exhibit a western boundary current and satisfy the no-penetration condition on the coast.}
\label{fig:stommel}
\end{figure}

\FloatBarrier
\subsection{Piecewise polynomial fields}\label{sec:ex-piecewise}
The following examples apply \Cref{sec:piecewise} to non-convex domains. The L-shaped domain illustrates how local tangency conditions and interface continuity determine the dimension of the piecewise space. The irregular hexagonal ring then examines the effect of higher-order matching on vorticity regularity and reconstruction error in a domain with two boundary components.

Consider the L-shaped domain with cells
\[
\mP_1=[0,1]^2,\qquad
\mP_2=[1,2]\times[0,1],\qquad
\mP_3=[0,1]\times[1,2].
\]
At degree $3$, the local exterior-tangential fields have the forms
\[
\begin{array}{c|c|c}
c&\xi_c&\text{coefficient spaces}\\\hline
1&(xa_1,yb_1)&a_1,b_1\in R_{\le2}\\
2&((x-2)a_2,y(y-1)b_2)&a_2\in R_{\le2},\ b_2\in R_{\le1}\\
3&(x(x-1)a_3,(y-2)b_3)&a_3\in R_{\le1},\ b_3\in R_{\le2}
\end{array}
\]
Their dimensions are $12,9,9$. Restricting differences to $x=1$ and $y=1$ gives a matrix of rank $12$, hence $\dim \PPt{0}{\Delta}_{\le3}=18$. Agreement at the remaining intersection $\mP_2\cap\mP_3=\{(1,1)\}$ follows through $\mP_1$.
\Cref{fig:piecewise} illustrates a fit to $14$ velocity observations per cell. 
\begin{figure}[htbp]
\centering
\includegraphics[width=0.62\linewidth]{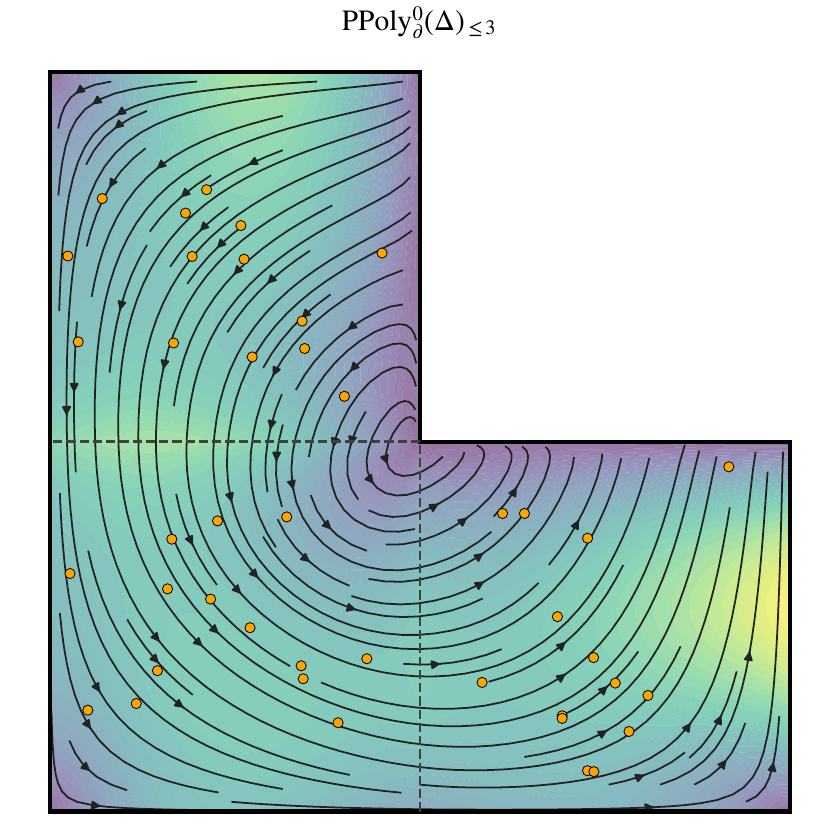}
\caption{A piecewise cubic fit on the L-shaped domain with matching order $r=0$. Solid lines mark $\partial\mQ$, dashed lines the two interior facets, and orange points the $42$ observations. Streamlines cross the interior interfaces while respecting the exterior boundary.}
\label{fig:piecewise}
\end{figure}

The preceding domains have only a few facet-normal directions, allowing substantial simplification of the local tangency conditions. We next consider a domain with more general facet orientations. Let $\mQ$ be the hexagon with vertices
\[
(0,1),\ (3,-1),\ (7,2),\ (6,6),\ (2,7),\ (-1,4),
\]
with the hexagonal hole
\[
(2,\tfrac94),\ (3,\tfrac32),\ (4,\tfrac52),\ (4,4),\ (\tfrac52,5),\ (\tfrac32,\tfrac72)
\]
removed, decomposed into the six quadrilaterals spanned by consecutive vertices of the two boundaries, as drawn in \Cref{fig:piecewise-annulus}. Its twelve exterior facets realize ten different normal directions, and the domain has no symmetry. It is also not simply connected: tangency applies to two boundary components, the outer hexagon and the boundary of the hole. 

Each cell has exactly two exterior facets, one on each component, and they are not parallel. An affine map carries any such pair to the coordinate axes, so each local model is a copy of the corner model of the L-shaped and
\[
\dim \Pext{\mP_c}_{\le k}=2\binom{k+1}{2},
\qquad
\sum_c\dim \Pext{\mP_c}_{\le k}=12\binom{k+1}{2}.
\]
The local dimension formula follows from this affine equivalence; the global dimensions below are obtained by computing the ranks of the matching conditions. Every shared face is a facet, one for each of the six pairs of neighboring cells, so no lower-dimensional matching condition arises here. The computed dimensions are
\[
\begin{array}{c|c|c|c|c|c}
k&\sum_c\dim \Pext{\mP_c}_{\le k}&\multicolumn{4}{c}{\dim\PPt{r}{\Delta}_{\le k}}\\
&&r=0&r=1&r=2&r=3\\\hline
3&72&24&0&0&0\\
4&120&60&12&0&0\\
5&180&108&48&0&0\\
6&252&168&96&36&0\\
7&336&240&156&84&24
\end{array}
\]
The facet geometry affects the degree at which nonzero smooth piecewise fields occur. A rectangular ring of comparable size has a nonzero $C^1$ space already at degree $3$; here the first nonzero $\PPt{1}{\Delta}_{\le k}$ is at degree $4$ and the first nonzero $\PPt{2}{\Delta}_{\le k}$ at degree $6$, because a matching condition across an interface in general position couples both components of the field instead of one.

At matching order $r=k$ the cells are forced to agree, so $\PPt{k}{\Delta}_{\le k}$ is the space of global fields of degree at most $k$ tangent to all twelve facets. Imposing the twelve divisibility conditions of \Cref{lem:1tangent} directly gives dimension $0$ for every $k\le9$, first $3$ at $k=10$ and then $15$ and $29$; the two computations agree at every degree where both were run. 

\Cref{fig:piecewise-annulus} fixes the domain, the degree $k=9$ and a set of $660$ noisy velocity observations of a flow circulating in the ring, and varies only $r$. The spaces for $r=0,1,2$ have dimensions $420$, $312$ and $216$, respectively. We compute the least-squares fits in locally scaled Legendre coordinates to avoid ill-conditioning from global monomial coordinates. All three fits are continuous and tangent to both boundary components.

Against the unperturbed reference at independent test points, the velocity errors for $r=0,1,2$ are $33.3\%$, $25.4\%$ and $22.4\%$ of the signal, respectively. For $r=2$, increasing the degree from $7$ to $9$ reduces the error from $39.4\%$ to $22.4\%$ on the same data, but degree $10$ gives $27.4\%$.
\begin{figure}[htbp]
\centering
\includegraphics[width=\linewidth]{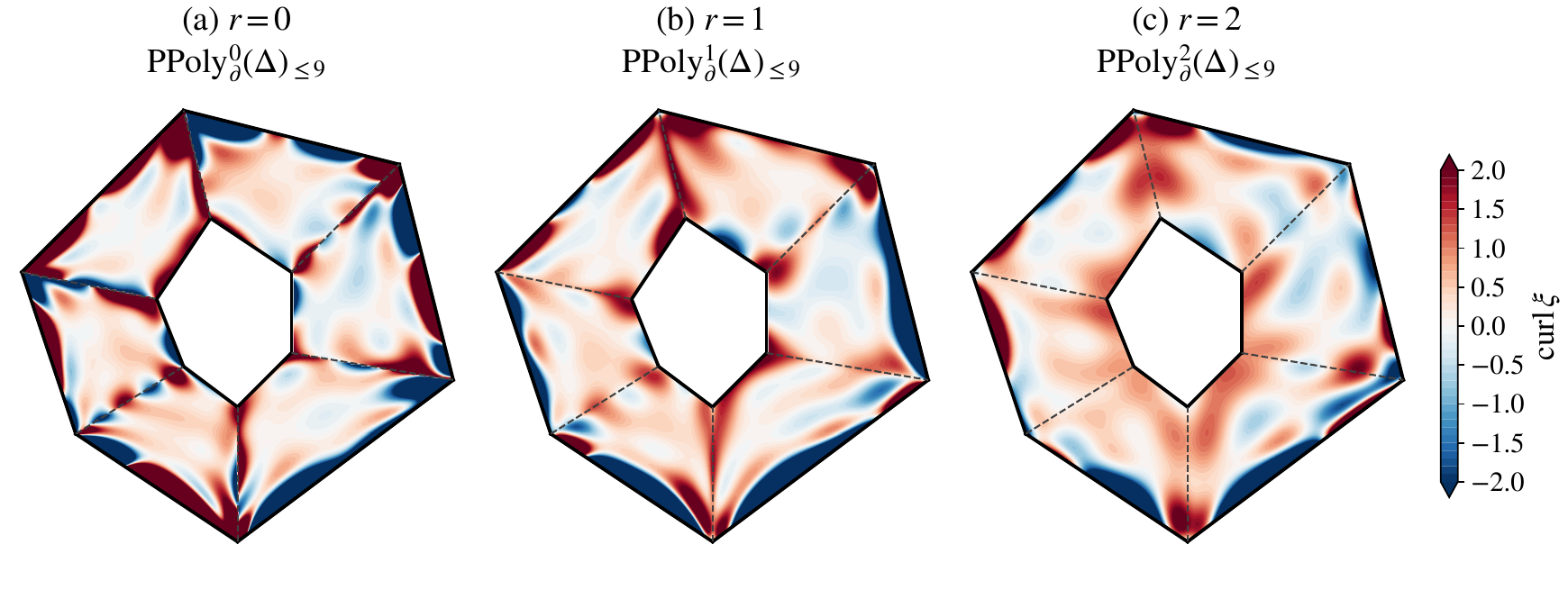}
\caption{Degree-$9$ piecewise fits on the irregular hexagonal ring from the same $660$ noisy observations, with matching order $r=0,1,2$. Color indicates the vorticity of each fit on a common scale; solid lines are the two boundary components and dashed lines the six interior facets.}
\label{fig:piecewise-annulus}
\end{figure}

\section{Conclusion}\label{sec:conclusion}
We have given an explicit algebraic characterization and computable bases for polynomial vector fields of bounded degree tangent to the boundary of a convex polyhedron.
This yields least-squares reconstruction with the no-penetration condition satisfied identically on every boundary facet.
For distinct velocity observations compatible with the boundary condition, exact interpolation is possible at a finite degree; in dimension at least two, the interpolant may additionally be chosen divergence-free.

The framework accommodates linear differential constraints and derivative observations, and extends to piecewise polynomial fields of prescribed smoothness on non-convex polyhedral domains.
For a fixed matching order, the piecewise construction retains the finite-degree interpolation guarantee for boundary-compatible velocity data.
Further work includes improving numerical conditioning at high degree, selecting the degree and matching order from the data, and extending the construction to curved boundaries.

\usepaperbibliographystyle
\bibliography{reference}

\end{document}